\documentclass[12pt]{article}

\usepackage{soul,tikz, adjustbox, genyoungtabtikz}
\usepackage{forest, tikz-qtree,pseudocode,ulem, framed}
\newsavebox{\tempbox}
\newcommand{\cbox}[2]{%
 \fcolorbox{black}{#1}{\texttt{#2\strut}}\kern-\fboxrule}
\usepackage{times}
\usepackage{palatino}
\usepackage{amssymb,amsmath,amsthm}
\usepackage{epsfig}
\usepackage{xcolor}

\newcommand{\smallspacing}{\baselineskip = 0.98\normalbaselineskip}

\newcommand{\grey}{!\Yfillcolor{gray!50}}
\newcommand{\blue}{!\Yfillcolor{blue!50}}
\newcommand{\white}{!\Yfillcolor{white}}
\newtheorem{example}{Example}
\newtheorem{theorem}{Theorem}
\newtheorem{lemma}{Lemma} 
 
\newtheorem{conjecture}{Conjecture} 
\def\AND{\wedge}
\def\OR{\vee}

\newcommand{\twozeros}{00}
\newcommand{\zone}{01}
\newcommand{\onezero}{10}
\newcommand{\eleven}{11}
\newcommand{\twelve}{12}
\newcommand{\thirteen}{13}
\newcommand{\fourteen}{14}

\newenvironment{proof-sketch}{\noindent{\bf Proof Sketch:}\hspace*{1em}}{\qed}

\long\def\greybox#1{%
    \newbox\contentbox%
    \newbox\bkgdbox%
    \setbox\contentbox\hbox to \hsize{%
        \vtop{
            \kern\columnsep
            \hbox to \hsize{%
                \kern\columnsep%
                \advance\hsize by -2\columnsep%
                \setlength{1.1 \textwidth}{\hsize}%
                \vbox{
                    \parskip=\baselineskip
                    \parindent=0bp
                    #1
                }%
                \kern\columnsep%
            }%
            \kern\columnsep%
        }%
    }%
    \setbox\bkgdbox\vbox{
        \pdfliteral{0.85 0.85 0.85 rg}
        \hrule width  \wd\contentbox %
               height \ht\contentbox %
               depth  \dp\contentbox
        \pdfliteral{0 0 0 rg}
    }%
    \wd\bkgdbox=0bp%
    \vbox{\hbox to \hsize{\box\bkgdbox\box\contentbox}}%
    \vskip\baselineskip%
}

\begin{document}

\pagestyle{empty}

\title{Partition into heapable sequences, heap tableaux and a multiset extension of Hammersley's process} 

\author{Gabriel Istrate, Cosmin  Bonchi\c{s}\footnote{Dept. of Computer Science, West University of Timi\c{s}oara,  Timi\c{s}oara, Romania. 
and e-Austria Research Institute, Bd. V. P\^{a}rvan 4, cam. 045
B, Timi\c{s}oara, RO-300223, Romania. Corresponding author's email: {\tt gabrielistrate@acm.org }}}

\date{}

\maketitle

\begin{abstract} 
We investigate partitioning of integer sequences into heapable subsequences (previously defined and established by Mitzenmacher et al.\\ \cite{byers2011heapable}). We show that an extension of patience sorting computes the decomposition into a minimal number of heapable subsequences (MHS). We connect this parameter to an interactive particle system, a multiset extension of Hammersley's process, and investigate its expected value  on a random permutation. In contrast with the (well studied) case of the longest increasing subsequence, we bring experimental evidence 
that the correct asymptotic scaling is $\frac{1+\sqrt{5}}{2}\cdot \ln(n)$. Finally we give a heap-based extension of Young tableaux, prove a hook inequality and an extension of the Robinson-Schensted correspondence. 
\end{abstract}

\section{Introduction}

{\it Patience sorting} \cite{mallows-patience} and {\it the longest increasing (LIS) sequence} are well-studied topics in combinatorics. The analysis  of the expected length of the LIS of a random permutation is a classical problem  displaying interesting connections with the theory of interacting particle systems \cite{aldous1999longest} and that of combinatorial Hopf algebras \cite{hivert2007introduction}.  Recursive versions of patience sorting are involved (under the name of {\it Schensted procedure} \cite{schensted}) in the theory of Young tableaux. A wonderful recent reference for the rich theory of the longest increasing sequences (and substantially more) is \cite{romik2014}. 

Recently  Mitzenmacher et al. \cite{byers2011heapable} introduced, under the name of {\it heapable sequence}, an interesting variation on the concept of increasing sequences. Informally, a sequence of integers is heapable if it can be successively inserted into a (not necessarily complete) binary tree satisfying the heap property without having to resort to node rearrangements. Mitzenmacher et al. showed that the longest heapable subsequence in a random permutation grows linearly (rather than asymptotically equal to $2\sqrt{n}$ as does LIS) and raised as an open question the issue of  extending the rich theory of LIS to the case of heapable sequences. 

In this paper we partly answer this open question: we define a family $MHS_{k}(X)$ of measures (based on decomposing the sequence into subsequences heapable into a min-heap of arity at most $k$) and show that a variant of patience sorting correctly computes the values of these parameters. We show that this family of measures forms an infinite hierarchy, and investigate the expected value of parameter $MHS_{2}[\pi]$, where $\pi$ is a random permutation of order $n$. Unlike the case $k=1$ where $E[MHS_{1}[\pi]]=E[LDS[\pi]]\sim 2\sqrt{n}$, we argue that in the case $k\geq 2$ the correct scaling is logarithmic, bringing experimental evidence that the precise scaling is $E[MHS_{2}[\pi]]\sim \phi\ln{n}$, where $\phi=\frac{1+\sqrt{5}}{2}$ is the golden ratio. The analysis exploits the connection with a new, multiset extension of the Hammersley-Aldous-Diaconis process \cite{aldous1995hammersley}, an extension that may be of independent interest. 
Finally, we introduce a heap-based generalization of Young tableaux. We prove (Theorem~\ref{unif} below) a hook inequality related to the hook formula for Young tableaux \cite{hook-formula} and Knuth's hook formula for heap-ordered trees \cite{knuth:vol3}, and (Theorem~\ref{rs-heap}) an extension of the Robinson-Schensted (R-S) correspondence. 

\section{Preliminaries}

For $k\geq1$ define alphabet $\Sigma_{k}=\{1,2,\ldots, k\}$. Define 
as well $\Sigma_{\infty}=\cup_{k\geq 1} \Sigma_{k}$. Given words $x,y$ over $\Sigma_{\infty}$ we will denote by $x\sqsubseteq y$ the fact that $x$ is a prefix of $y$. The set of (non-strict) prefixes of $x$ will be denoted by $Pref(x)$. Given words $x,y\in \Sigma_{\infty}^{*}$ define the {\it prefix partial order} $x\preceq_{ppo} y$ as follows: If $x\sqsubseteq y$ then $x\preceq_{ppo} y$. If $x=za$, $y=zb$, $a,b\in \Sigma_{\infty}$ and $a<b$ then $x\preceq_{ppo} y$. $\preceq_{ppo}$ is the transitive closure of these two constraints. Similarly, the {\it lexicographic partial order} $\preceq_{lex}$ is defined as follows: If $x\sqsubset y$ then $x\preceq_{lex} y$. If $x=za$, $y=zb$, $a,b\in \Sigma_{\infty}$ and $a<b$ then $x\preceq_{lex} y$. $\preceq_{lex}$ is the transitive closure of these two constraints.

A {\it $k$-ary tree} is a finite, $\preceq_{ppo}$-closed set $T$ of words over alphabet $\Sigma_{k}=\{1,2,\ldots, k\}$. That is, 
we impose the condition that positions on the same level in a tree are filled preferentially from left to right. 
The {\it position $pos(x)$ of node $x$ in a $k$-ary tree} is the string over alphabet $\{1,2,\ldots, k\}$ encoding the path from the root to the node (e.g. the root has position $\lambda$, its children have positions $1,2,\ldots, k$, and so on). 
A {\it $k$-ary (min)-heap} is a function $f:T\rightarrow {\bf N}$ monotone with respect to $pos$, i.e. $
(\forall x,y\in T), [pos(x)\sqsubseteq pos(y)] \Rightarrow [f(x)\leq f(y)].$

A {\it (binary min-)heap} is a binary tree, not necessarily complete, such that $A[parent[x]]\leq A[x]$ for every non-root node $x$. If instead of binary we require the tree to be $k$-ary we get the concept of $k$-ary min-heap.

 A  sequence $X=X_{0},\ldots, X_{n-1}$ is {\it $k$-heapable} if there exists some $k$-ary tree $T$ whose nodes are 
 labeled with (exactly one of) the elements of $X$, such that for every non-root node $X_{i}$ and parent $X_{j}$, $X_{j}\leq X_{i}$ and $j<i$. In particular a 2-heapable sequence will simply be called {\it heapable}  \cite{byers2011heapable}. Given sequence of integer numbers $X$, denote by $MHS_{k}(X)$ {\it the smallest number of heapable (not necessarily contiguous) subsequences} one can decompose $X$ into. $MHS_{1}(X)$ is equal \cite{levcopoulos-petersson-sus} to the {\it shuffled up-sequences (SUS)} measure in the theory of presortedness. 

\begin{example}
Let $X=[2,4,3,1]$. Via patience sorting  $MHS_{1}(X)=SUS(X)$ $=3$. $MHS_{2}(X)=2$, since subsequnces $[2,4,3]$ and $[1]$ are 2-heapable. On the other hand, for every $k\geq 1$, $MHS_{k}([k,k-1,\ldots, 1])=k$. 
\end{example} 

Analyzing the behavior of LIS relies on the correspondence between longest increasing sequences and an interactive particle system \cite{aldous1995hammersley} called the  {\it Hammersley-Aldous-Diaconis (shortly, Hammersley or HAD) process}. We give it the  multiset generalization displayed in Figure~\ref{ham:k}. Technically, to recover the usual definition of Hammersley's process one should take $X_{a}>X_{t+1}$ (rather than $X_{a}<X_{t+1}$). This small difference arises since we want to capture $MHS_{k}(\pi)$, which generalizes $LDS(\pi)$, rather than $LIS(\pi)$ (captured by Hammersley's process). This slight difference is, of course, inconsequential: our definition is simply the "flipped around the midpoint of segment [0,1]" version of such a generalization, and has similar behaviour).

\begin{figure}
\begin{framed} 
\begin{itemize} 
\item A number of individuals appear (at integer times $i\geq 1$) as random numbers $X_{i}$, uniformly distributed in the interval $[0,1]$. 
\item Each individual is initially endowed with $k$ "lifelines". 
\item The appearance of a new individual $X_{t+1}$ subtracts a life from the largest individual $X_{a}<X_{t+1}$ (if any) still alive at moment $t$. 

\end{itemize}
\end{framed} 

\caption{$HAD_{k}$, the multiset Hammersley process with $k$ lifelines.}
\label{ham:k}
\end{figure}

\section{A greedy approach to computing $MHS_{k}$}

First we show that one can combine patience sorting and the greedy approach in  \cite{byers2011heapable} to 
obtain an algorithm for computing $MHS_{k}(X)$. To do so, we must adapt to our purposes some notation in that paper. 

A binary tree with $n$ nodes has $n+1$ positions (that will be called {\it slots}) where one can add a new number. We will identify a slot with the minimal value of a number that can be added to that location. For heap-ordered trees it is the value of the parent node. Slots easily generalize to forests. The number of slots of a forest with $d$ trees and $n$ nodes is $n+d$. 

Given a binary heap forest $T$, the {\it signature of $T$} denoted $sig(T)$, is the vector of the (values of) free slots in $T$, in sorted (non-decreasing) order. Given two binary heap forests $T_{1},T_{2}$, {\it $T_{1}$ dominates $T_{2}$} if $|sig_{T_{1}}|\leq |sig_{T_{2}}|$ and inequality $sig_{T_{1}}[i]\leq sig_{T_{2}}[i]$ holds for all $1\leq i \leq |sig_{T_{1}}|$.

\begin{theorem} 
For every fixed $k\geq 1$ there is a polynomial time algorithm that, given sequence $X=(X_{0},\ldots, X_{n-1})$ as input, computes $MHS_{k}(X)$. 
\end{theorem}

\begin{proof} 

We use the greedy approach of Algorithm~\ref{fig:greedy}. Proving  correctness of the algorithm employs the following

\begin{center}
\begin{pseudocode}[doublebox]{GREEDY}{W}
\mbox{INPUT } W=(w_{1},w_{2},\ldots, w_{n})\mbox{ a list of integers.} \\
\mbox{Start with empty heap forest }T=\emptyset. \\
\mbox{for }i\mbox{ in range(n):} \\
	\hspace{5mm}\IF \mbox{(there exists a slot where }X_{i}\mbox{ can be inserted):}\\
	\hspace{20mm}\mbox{ insert }X_{i}\mbox{ in the slot with the lowest value }
	\hspace{20mm} \ELSE :\\
	\hspace{20mm}\mbox{ start a new heap consisting of }X_{i}\mbox{ only.} 
\label{fig:greedy}
\end{pseudocode}
\end{center}

\begin{lemma} 
Let $T_{1},T_{2}$ be two heap forests  such that $T_{1}$ dominates $T_{2}$. Insert a new element $x$ in both $T_{1}$ and $T_{2}$: greedily in $T_{1}$ (i.e. at the largest slot with value less or equal to $x$, or as the root of a new tree, if no such slot exists) and  arbitrarily in $T_{2}$, obtaining forests $T_{1}^{\prime},T_{2}^{\prime}$, respectively. 
Then $T_{1}^{\prime}$ dominates $T_{2}^{\prime}$.
\label{dom}
\end{lemma} 
\begin{proof} 
First note that, by domination, if no slot of $T_{1}$ can accomodate $x$ (which, thus, starts a new tree) then a similar property is true in $T_{2}$ (and thus $x$ starts a new tree in $T_{2}$ as well). 

Let $sig_{T_{1}}=(a_{1},a_{2},\ldots )$ and $sig_{T_{2}}=(b_{1},b_{2},\ldots )$ be the two signatures.
The process of inserting $x$ can be described as adding two copies of $x$ to the signature of $T_{1}(T_{2})$ and (perhaps) removing a label $\leq x$ from the two signatures. The removed label is $a_{i}$, the largest label $\leq x$, in the case of greedy insertion into $T_{1}$. Let $b_j$ be the largest value (or possibly none) in $T_{2}$ less or equal to $x.$ Some $b_k$ less or equal to $b_j$ is replaced by two copies of $x$ in $T_2.$ The following are true: 
\begin{itemize} 
\item The length of $sig_{T_{1}^{\prime}}$ is at most that of $sig_{T_{2}^{\prime}}$.
\item The element $b_k$ (if any) deleted by $x$ from $T_{2}$ satisfies $b_k\leq x$. Its index in $T_2$ is less or equal to $i$. 
\item The two $x$'s are inserted to the left of the deleted  (if any) positions in both $T_{1}$ and $T_{2}$. 
\end{itemize}

Consider some position $l$ in $sig_{T_{1}^{\prime}}$. Our goal is to show that $a^{\prime}_{l}\leq b^{\prime}_{l}$. Several cases are possible: 
\begin{itemize} 
\item $l < k$. Then $a^{\prime}_{l}= a_{l}$ and $b^{\prime}_{l}= b_{l}$. 
\item $k \leq l < j$. Then $a^{\prime}_{l}= a_{l}$ and $b^{\prime}_{l}= b_{l+1} \geq a_{l}$. 
\item $j\leq l\leq i+k-1$. Then $a^{\prime}_{l}\leq x$ and $b^{\prime}_{l}\geq x$. 
\item $l > i+k-1$. Then $a^{\prime}_{l} = a_{l-k+1}$ and $b^{\prime}_{l}= b_{l-k+1}$. 
\end{itemize}

\begin{figure}
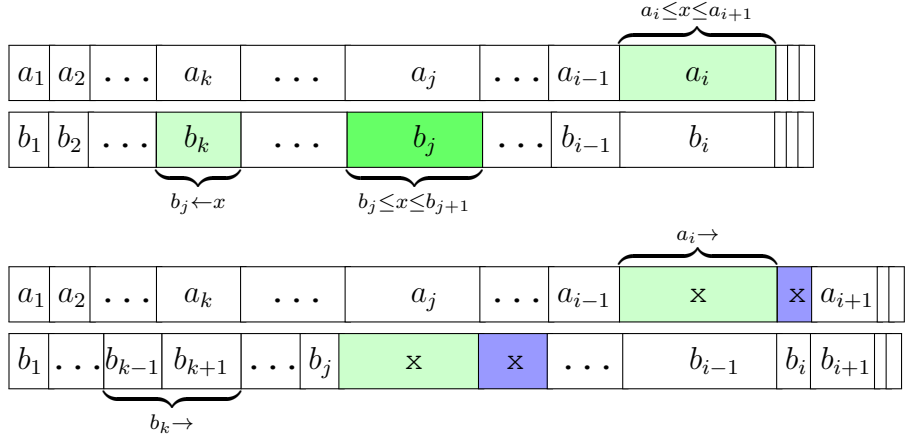
 
\begin{eqnarray*}
&& \cbox{white}{$a_1$}\!
  \cbox{white}{$a_2$}\!
  \cbox{white}{\ldots\hspace{-0.03cm}}\!
  \cbox{white}{~$a_k$~\hspace{0.06cm}}\!
  \cbox{white}{~\ldots~\hspace{-0.03cm}}\!
  \cbox{white}{~~~$a_j$~~}\!
  \cbox{white}{\ldots}\!
  \cbox{white}{$a_{i-1}$}\!
  \overbrace{\cbox{green!20}{~~~$a_i$~~~}}^{a_{i}\leq x\leq a_{i+1}}\!
  \cbox{white}{}\!
  \cbox{white}{}\!
  \cbox{white}{\hspace{-0.02cm}}\\
&& \cbox{white}{$b_1$\hspace{0.03cm}}\!
  \cbox{white}{$b_2$\hspace{0.03cm}}\!
  \cbox{white}{\ldots\hspace{-0.08cm}}\!
  \underbrace{\cbox{green!20}{~$b_k$~\hspace{0.04cm}}}_{b_j \leftarrow x}\!
  \cbox{white}{~\ldots~\hspace{-0.08cm}}\!
  \underbrace{\cbox{green!60}{~~~$b_j$~~\hspace{-0.02cm}}}_{b_j \leq x \leq b_{j+1}}\!
  \cbox{white}{\hspace{0.02cm}\ldots\hspace{-0.02cm}}\!
  \cbox{white}{\hspace{0.02cm}$b_{i-1}$\hspace{0.06cm}}\!
  \cbox{white}{\hspace{0.04cm}~~~$b_i$~~~\hspace{0.05cm}}\!
  \cbox{white}{}\!
  \cbox{white}{}\!
  \cbox{white}{\hspace{-0.02cm}}\! \\
\vspace{-0.6cm}
&& \cbox{white}{$a_1$}\!
  \cbox{white}{$a_2$}\!
  \cbox{white}{\ldots\hspace{-0.03cm}}\!
  \cbox{white}{~$a_k$~\hspace{0.06cm}}\!
  \cbox{white}{~\ldots~\hspace{-0.03cm}}\!
  \cbox{white}{~~~$a_j$~~}\!
  \cbox{white}{\ldots}\!
  \cbox{white}{$a_{i-1}$}\!
  \overbrace{\cbox{green!20}{\hspace{0.05cm}~~~x~~~\hspace{0.05cm}}}^{a_i\rightarrow}\!
  \cbox{blue!40}{\hspace{0.03cm}x\hspace{0.02cm}}\!
  \cbox{white}{$a_{i+1}$}\!
  \cbox{white}{}\!
  \cbox{white}{\hspace{-0.02cm}}\\
&& \cbox{white}{$b_1$\hspace{0.03cm}}\!
  \cbox{white}{\hspace{-0.1cm}\ldots\hspace{-0.15cm}}\!
  \underbrace{\cbox{white}{\hspace{-0.1cm}$b_{k-1}$\hspace{-0.09cm}} \cbox{white}{\hspace{0.04cm}$b_{k+1}$\hspace{0.05cm}}}_{b_k \rightarrow }\!
  \cbox{white}{\hspace{-0.07cm}\ldots\hspace{-0.06cm}}\!
  \cbox{white}{$b_j$\hspace{0.03cm}}\!     
  \cbox{green!20}{~~~x~~~\hspace{-0.07cm}}\!
  \cbox{blue!40}{~x~}\!
  \cbox{white}{\hspace{0.04cm}\ldots\hspace{0.05cm}}\!
  \cbox{white}{\hspace{0.01cm}~~~$b_{i-1}$~~\hspace{-0.06cm}}\!
  \cbox{white}{$b_i$}\!
  \cbox{white}{\hspace{0.02cm}$b_{i+1}$}\!
  \cbox{white}{}\!
  \cbox{white}{\hspace{-0.02cm}}\\
\end{eqnarray*}
\caption{The argument of Lemma~\ref{dom}. Pictured vectors (both initial and resulting) have equal lengths (which may not always be the case).}
\end{figure} 
\end{proof}
\qed

Let $X$ be a sequence of integers, OPT be an optimal partition of $X$ into $k$-heapable sequences  and $\Gamma$ be the solution produced by GREEDY. 
Applying Lemma~\ref{dom} repeatedly we infer that whenever GREEDY adds a new heap the same thing happens in OPT. Thus the number of heaps created by Greedy is optimal, which means that the algorithm computes $MHS_{k}(X)$. \qed 
\end{proof}

Trivially $MHS_{k}(X)\le MHS_{k-1}(X)$. On the other hand

\begin{theorem} \label{ineq}
 The following statements (proved in the Appendix) are true for every $k\geq 2$: (a). there exists a sequence $X$ such that $MHS_{k}(X)<MHS_{k-1}(X)$ $<\ldots <MHS_{1}(X)$; (b). $
\sup\limits_{X}\mbox{ } [MHS_{k-1}(X)-MHS_{k}(X)] =\infty$. 
\end{theorem}

\section{The connection with the multiset Hammersley process} 

Denote by $MinHAD_{k}(n)$ the random variable denoting {\it the number of times $i$ in the evolution of process $HAD_{k}$ up to time $n$ when the newly inserted particle $X_{i}$ has lower value than all the existing particles at time $i$.}
The observation from \cite{hammersley1972few,aldous1995hammersley} generalizes to: 
\begin{theorem} 
For every fixed $k,n\geq 1$ $E_{\pi\in S_{n}} [MHS_{k}(\pi)]=E[MinHAD_{k}(n)]$. 
\end{theorem}
\begin{proof-sketch}
W.h.p. all $X_{i}$'s are different. We will thus ignore in the sequel the opposite alternative. Informally minima correspond to new heaps and live particles to slots in these heaps (cf. also Lemma~\ref{dom}). 
\end{proof-sketch}

\section{The asymptotic behavior of $E[MHS_{2}[\pi]]$}

The asymptotic behavior of $E[MHS_{1}[\pi]]$ where $\pi$ is a random permutation in $S_{n}$ is a classical problem in probability theory: results in \cite{hammersley1972few}, \cite{logan1977variational}, \cite{vershik1977asymptotics}, \cite{aldous1995hammersley} show that it is asymptotically equal to $2\sqrt{n}$. 

A simple lower bound valid for all values of $k\geq 1$ is 

\begin{theorem}
For every fixed $k,n\geq 1$ 
\begin{equation} 
E_{\pi\in S_{n}} [MHS_{k}(\pi)]\geq H_{n}\mbox{, the }n\mbox{'th harmonic number.    }
\end{equation} 
\end{theorem} 
\begin{proof}
For $\pi\in S_{n}$ the set of its {\it minima} is defined as $Min(\pi)=\{j\in [n]: \pi[j]<\pi[i]\mbox{ for all } 1\leq i<j\}$ (and similarly for maxima). It is easy to see that $MHS_{k}[\pi]\geq |Min[\pi]|$. Indeed, every minimum of $\pi$ must determine the starting of a new heap, no matter what $k$ is. 
Now we use the well-known formula $E_{\pi\in S_{n}}[|Min[\pi]|]=E_{\pi\in S_{n}}[|Max[\pi]|]=H_{n}$ \cite{knuth:vol3}. \end{proof}\qed

To gain insight in the behavior of process $HAD_{2}$ we note that, rather than giving the precise values of $X_{0},X_{1},\ldots, X_{t}\in [0,1]$, an equivalent random model inserts $X_{t}$ uniformly at random in any of the $t+1$ possible positions determined by $X_{0},X_{1},\ldots, X_{t-1}$. This model translates into the following equivalent  combinatorial description of $HAD_{k}$: word $w_{t}$ over the alphabet $\{-1,0,1,2\}$ describes the state of the process at time $t$.  Each $w_{t}$ conventionally starts with a $-1$ and continues with a sequence of $0$, $1$'s and $2$'s,  informally the "number of lifelines" of particles at time $t$. For instance $w_{0}=0$, $w_{1}=02$, $w_{2}$ is either $022$ or $012$, depending on $X_{0} <> X_{1}$, and so on. At each time $t$ a random letter of $w_{t}$ is chosen (corresponding to a position for $X_{t}$) and we apply one of the following transformations, the appropriate one for the chosen position: 
\begin{itemize} 
\item {\it Replacing $-10^{r}$  by $-10^{r}2$:} This is the case when $X_{t}$ is the smallest particle still alive, and to its right there are $r\geq 0$ dead particles.
\item {\it Replacing $10^{r}$ by $0^{r+1}2$:} Suppose that $X_{a}$ is the largest live label less or equal to $X_t$, that the corresponding particle $X_{a}$ has one lifetime at time $t$, and that there are $r$ dead particles between $X_{a}$ and $X_{t}$. Adding $X_{t}$ (with multiplicity two) decreases multiplicity of $X_{a}$ to 0.  
\item {\it Replacing $20^{r}$ by $10^{r}2$:} Suppose that $X_{a}$ is the largest label less or equal to $X_t$, its multiplicity is two, and there are $r\geq 0$ dead particles between $X_{a}$ and $X_{t}$. Adding $X_{t}$ removes one lifeline from particle $X_{a}$.  
\end{itemize}  

Simulating the (combinatorial version of the) Hammersley process with two lifelines confirms the fact that $E[MHS_{2}(\pi)]$ grows significantly slower than $E[MHS_{1}(\pi)]$: The $x$-axis in the figure is logarithmic. The scaling is clearly different, and is consistent (see inset) with logarithmic growth (displayed as a straight line on a plot with log-scaling on the $x$-axis). Experimental results (see the inset/caption of Fig.~\ref{k:5}) suggest the following bold
\begin{conjecture} \label{main} We have $
\lim_{n\rightarrow \infty}\frac{E[MHS_{2}[\pi]]}{\ln(n)}=\phi$, with $\phi=\frac{1+\sqrt{5}}{2}$ the golden ratio.
More generally, for an arbitrary $k\geq 2$ the relevant scaling is 
\vspace{-0.3cm}
\begin{equation} 
\lim_{n\rightarrow \infty}\frac{E[MHS_{2}[\pi]]}{\ln(n)}=
 \frac{1}{\phi_{k}},
\label{general:k}
\end{equation}
\vspace{-0.3cm} 
 where $\phi_{k}$ is the unique root in $(0,1)$ of equation $X^{k}-X^{k-1}+\ldots +X=1$. 
\end{conjecture}

\begin{figure} 
\begin{center}
\includegraphics[width=10cm,height=6cm]{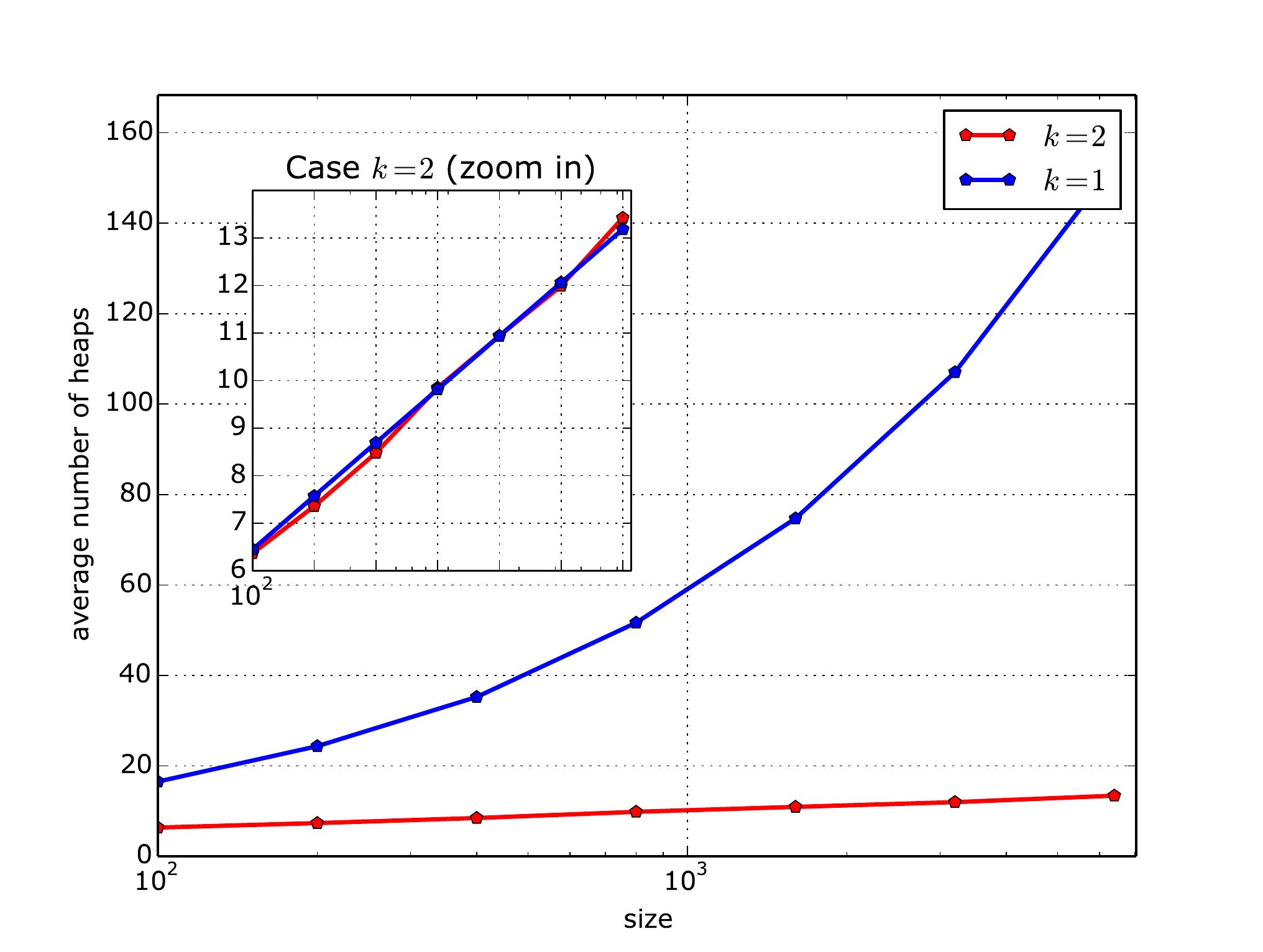}
\end{center}
\caption{Scaling of expected value of $MHS_{k}[\pi]$ for $k=1,2$. The inset shows $E[MHS_{2}[\pi]]$ (red) versus $\phi\cdot  \ln(n)+1$ (blue). The fit is strikingly accurate.} 
\label{k:5}
\end{figure} 

We plan to present the experimental evidence for the truth of equation~(\ref{general:k}) and a nonrigorous, "physics-like" justification, together with further insights on the so-called {\it hydrodynamic behavior} \cite{groenenboom} of the $HAD_{k}$ process in subsequent work \cite{hydrodynamic}. 
For now we limit ourselves to showing that one can (rigorously) perform a first step in the analysis of the $HAD_{2}$ process: we prove convergence of (some of) its structural characteristics. This will likely be useful in a full rigorous proof of Conjecture~\ref{main}. 

Denote by $L_{t}$ the number of digits 1+2, and by $C_{t}$ the number of ones in $w_{t}$. Let $l(t)=E[\frac{L(t)}{t}]$, $c(t)=E[\frac{C(t)}{t}]$.  $l(t),c(t)$ always belong to $[0,1]$. 

\begin{theorem} There exist constants $l,c\in [0,1]$ such that 
$l(t)\rightarrow l$, $c(t)\rightarrow c$. 
\label{limit}
\end{theorem} 

\begin{proof-sketch}
We use a standard tool, {\it subadditivity}: if sequence $a_{n}$ satisfies $a_{m+n}\leq a_{m}+a_{n}$ for all $m,n\geq 1$ then (by Fekete's Lemma (\cite{steele1997probability} pp. 3, \cite{szpankowski}) $\lim_{n\rightarrow \infty} a_{n}/n$ exists. We show in the Appendix that this is the case for two independent linear combinations of $l(t)$ and $c(t).$\end{proof-sketch}

Experimentally (and nonrigorously) $l=\phi-1=\frac{\sqrt{5}-1}{2}$ and $c=\frac{3-\sqrt{5}}{2}$. "Physics-like" nonrigorous arguments then imply the desired scaling. An additional ingredient is that  digits 0/1/2 are uniformly distributed (conditional on their density) in a large  $w_{t}$.
This is intuitively true since for large $t$ the behavior of the $HAD_{k}$ process is described by a compound Poisson process. We defer  more complete explanations to \cite{hydrodynamic}.

\section{Heap tableaux, a hook inequality and a generalization of the Robinson-Schensted Correspondence.} 

Finally, we present an extension of Young diagrams to heap-based ta-\\ bleaux. All proofs are given in the Appendix. A {\it ($k$-)heap tableau $T$}  is $k$-ary min-heap of integer vectors, so that for every $r\in \Sigma_{k}^{*}$, the vector $V_{r}$ at address $r$ is nondecreasing.
We formally represent the tableau as a function $T:\Sigma_{k}^{*} \times {\bf N}   \rightarrow {\bf N} \cup \{\perp\}$ such that (a).  $T$ has {\it finite support:} the set $dom(T)=\{(r,a):T(r,a)\neq \perp\}$ of nonempty positions is finite. 
(b). $T$ is {\it $\sqsubseteq$-nondecreasing}: if $T(r,a)\neq \perp$ and $q \sqsubset r$ then $T(q,a)\neq \perp$ and $T(q,a)\leq T(r,a)$. In other words, $T(\cdot,a)$ is a min-heap. (c).  $T$ is {\it columnwise increasing}: if $T(r,a)\neq \perp$ and $b< a$ then $T(r,b)\neq \perp$ and $T(r,b)< T(r,a)$. That is, each column $V_r$ is increasing.
The {\it shape of $T$} is the heap $S(T)$ where  node with address $r$ holds value $|V_r|$. 

A tableau is {\it standard} if (e). for all $1\leq i\leq n=|dom(T)|$, $|T^{-1}(i)|=1$ and (f). If $x\leq_{lex} y$ and $T(y,1)\neq \perp$ then $\perp \neq T(x,1)\leq T(y,1)$. I.e., labels in the first heap $H_{1}$ are increasing from left to right and top to bottom.

\begin{example} 
A heap tableau $T_{1}$ 
with 9 elements 
is presented in Fig.~\ref{forest-young} (a) and as a Young-like diagram in Fig.~\ref{forest-young} (b). Note that: (i). Columns correspond to  {\it rows} of $T_{1}$ (ii). Their labels are in  $\Sigma_{2}^{*}$, rather than {\bf N}. (iii). Cells may contain $\perp$. (iv). Rows need not be increasing, only {\it min-heap ordered}.  
\end{example} 
One important drawback of our notion of heap tableaux above is that they do not reflect the evolution of the process $HAD_{k}$ the way ordinary Young tableaux do (on their first line) for process $HAD_{1}$ via the Schensted procedure \cite{schensted}: A generalization with this feature would seem to require that each cell contains not an integer but a {\it multiset} of integers. Obtaining such a notion of tableau 
is part of ongoing research. 

However, we can motivate our definition of heap tableau by the first application below, a hook inequality for such tableaux. To explain it, note that heap tableaux generalize both heap-ordered trees and Young tableaux. In both cases there exist hook formulas that count the number of ways to fill in a structure with $n$ cells by numbers from $1$ to $n$: 
\cite{hook-formula} for Young tableaux  and \cite{knuth:vol3} (Sec.5.1.4, Ex.20) for heap-ordered trees. It is natural to wonder whether there exists a hook formula for heap tableaux
that provides a common generalization of both these results. 

Theorem~\ref{unif} gives a partial answer: not a formula but a {\it lower bound.} To state it, given $(\alpha,i)\in dom(T)$, define the {\it hook length $H_{\alpha,i}$} to be the cardinal of set $\{(\beta,j)\in dom(T): [(j=i)\AND (\alpha\sqsubseteq \beta)] \OR [(j\geq i) \AND (\alpha = \beta)]\}$. For example, Fig.~\ref{forest-young}(c). displays the hook lengths of cells in $T_{1}$. 
\begin{theorem} 
Given $k\geq 2$ and a $k$-shape $S$ with $n$ free cells, the number of ways to create a heap tableau $T$ with shape $S$ by filling its cells with numbers $\{1,2,\ldots, n\}$ is {\it at least} $
\frac{n!}{\prod_{(\alpha,i)\in dom(T)} H_{\alpha,i}}.$ The bound is tight for Young tableaux \cite{hook-formula}, 
heap-ordered trees \cite{knuth:vol3}, and infinitely many other examples, but is also {\bf not} tight for infinitely many (counter)examples. 
\label{unif} 
\end{theorem} 

We leave open the issue whether one can tighten up the lower bound above to a formula by modifying the definition of the hook length $H_{\alpha,i}$.

\begin{figure}[ht]
\begin{minipage}{.4\textwidth}
\begin{center} 
\begin{tikzpicture}
\node[inner sep=1pt] (A) at (3.5,0)
    {\gyoung(2;4;8:<[3]>)};
\node[inner sep=0pt] (B) at (2,-1)
    {\gyoung(\eleven;<12>;<13>:<[3]>)};
\node[inner sep=0pt] (C) at (5,-1)
    {\gyoung(<6>;<14>:<[2]>)};
\node[inner sep=0pt] (D) at (4,-2)
    {\gyoung(<10>:<[1]>)};
\draw[-,thick] (A.south) -- (B.north);
\draw[-,thick] (A.south) -- (C.north);
\draw[-,thick] (C.south) -- (D.north);
\end{tikzpicture}
\end{center}
\end{minipage}
\begin{minipage}{.6\textwidth}
\begin{center} 
\[
\gyoung(::\lambda:0:1:\twozeros:\zone:\onezero:::\lambda:0:1:\twozeros:\zone:\onezero:,:1;2;\eleven;6;\perp;\perp;\onezero::1;6;3;3;\perp;\perp;1:,:2;4;\twelve;\fourteen:::::2;4;2;1:,:3;8;\thirteen::::::3;2;1)
\]
\end{center}
\end{minipage}  

\caption{(a). Heap tableau $T_{1}$ and its shape $S(T_{1})$ (in brackets) (b). The equivalent Young tableau-like representation of $T_{1}$ and (c). The hook lengths.}
\label{forest-young}
\end{figure}
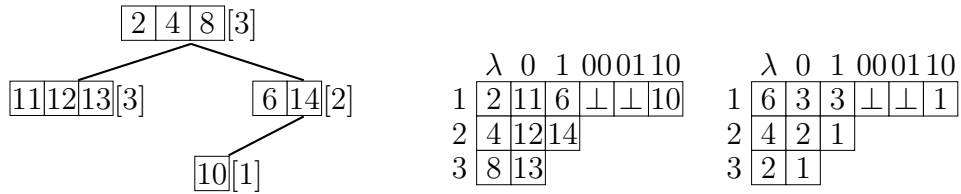

We can create $k$-heap tableaux from integer sequences by a version of the Schensted procedure \cite{schensted}. Algorithm Schensted-HEAP$_{k}$ below performs {\it column insertions}  and gives to any bumped element $k$ choices for insertion/bumping, the children of vector $V_r,$ with addresses $r\cdot \Sigma_{k}$.

\begin{theorem}
The result of applying the Schensted-HEAP$_{k}$ procedure to an  arbitrary permutation $X$ is indeed a $k$-ary heap tableau.  
\label{sch-heap} 
\end{theorem} 
\begin{example}
Suppose we start with $T_{1}$ from Fig.~\ref{forest-young}(a). Then (Fig.~\ref{insert}) 9 is appended to vector $V_\lambda$. 7 arrives, bumping 8, which in turn bumps 11. Finally 11 starts a new vector at position 00. Modified cells are grayed. 
\end{example}

Procedure Schensted-HEAP$_{k}$ does {\bf not} help in computing the longest heapable subsequence: The complexity of computing this parameter is open \cite{byers2011heapable}, and we make no progress on this issue. On the other hand, we can give a $k\geq 2$ version of the R-S correspondence:

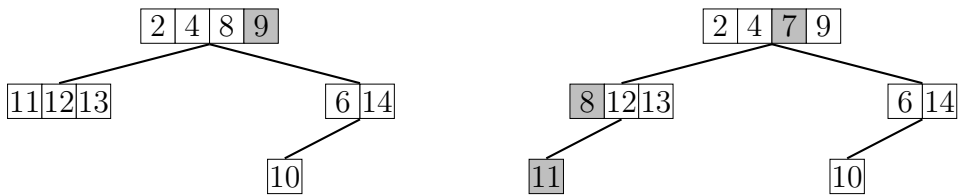
\begin{figure}[ht]
\begin{minipage}{.5\textwidth}
\begin{tikzpicture}
\node[inner sep=0pt] (A) at (2,1)
    {\gyoung(2;4;8!\grey<9>)};
\node[inner sep=0pt] (B) at (0,0)
    {\gyoung(\eleven;<12>;<13>)};
\node[inner sep=0pt] (C) at (4,0)
    {\gyoung(<6>;<14>)};
\node[inner sep=0pt] (D) at (3,-1)
    {\gyoung(<10>)};
\draw[-,thick] (A.south) -- (B.north);
\draw[-,thick] (A.south) -- (C.north);
\draw[-,thick] (C.south) -- (D.north);
\end{tikzpicture}
\end{minipage}
\begin{minipage}{.5\textwidth}
\begin{tikzpicture}
\node[inner sep=0pt] (A) at (4,1)
    {\gyoung(2;4!\grey<7>!\white<9>)};
\node[inner sep=0pt] (B) at (2,0)
    {\gyoung(!\grey<8>!\white;<12>;<13>)};
\node[inner sep=0pt] (C) at (6,0)
    {\gyoung(<6>;<14>)};
\node[inner sep=0pt] (D) at (5,-1)
    {\gyoung(<10>)};
\node[inner sep=0pt] (E) at (1,-1)
    {\gyoung(!\grey;<\eleven>)};
\draw[-,thick] (A.south) -- (B.north);
\draw[-,thick] (A.south) -- (C.north);
\draw[-,thick] (C.south) -- (D.north);
\draw[-,thick] (B.south) -- (E.north);
\end{tikzpicture}
\end{minipage} 
\caption{Inserting 9 and 7 into $T_{1}$.} 
\label{insert}
\end{figure}

\begin{theorem} For every $k\geq 2$ there exists a bijection between permutations $\pi\in S_{n}$ and pairs $(P,Q)$ of $k$-heap tableaux with $n$ elements and identical shape, where 
$Q$ is a standard tableau. 
\label{rs-heap}
\end{theorem} 
Condition "$Q$ is standard" is specific to case $k\geq 2$: heaps simply have "too many degrees of freedom" between siblings. Schensted-HEAP$_{k}$ solves this problem by starting new vectors from left to right and top to bottom. 

\begin{center} 
\begin{pseudocode}[doublebox]{Schensted-HEAP$_k$}{X=x_{0},\ldots, x_{n-1}} 
\FOR i\mbox{ in range}(n): BUMP(x_{i}, \lambda)\\
\\
 \mbox{\bf PROCEDURE BUMP}(x, S):  ~~~~~\# S \textit{ is a set of adresses}.\\	
\hspace{3mm}\mbox{- Attempt to append } x \mbox{ to some } V_r, r \in S\mbox{ (perhaps creating it)}\\
\hspace{5mm}\mbox{ (choose the first } r \mbox{ where appending $x$ keeps } V_r \mbox{ increasing).}\\
\hspace{3mm}\IF \mbox{ (this is not possible for {\bf any} vector }  V_r, r \in S): \\
\hspace{10mm}\mbox{- Let }B_{x}\mbox{ be the set of elements of value }>x,\\
\hspace{10mm}\mbox{ in all vectors $V_r,$ $r\in S$ (clearly } B_{x}\neq \emptyset\mbox{)}\\
\hspace{10mm}\mbox{- Let } y = min\{B_{x}\}\mbox{ and }r\mbox{ the address of its vector.}\\
\hspace{10mm} \mbox{- Replace }y\mbox{ by }x\mbox{ into }V_{r}\\
\hspace{10mm}\mbox{- } BUMP(y, r\cdot \Sigma_{k}) ~~~~~~~~~~~~\# \textit{bump $y$ into some child of $r$}
\label{sg}
\end{pseudocode} 
\end{center}

\section{Conclusion and Acknowledgments}

Our paper raises a large number of open issues. We briefly list a few: Rigorously justify Conjecture~\ref{main}. Study process $HAD_{k}$ and its variants \cite{leticia,cator-groenenboom}. Reconnect the theory to the analysis of {\it secretary problems} \cite{archibald2009hiring,broder2009hiring}. Find the distribution of $MHS_{k}[\pi]$. Obtain a hook formula. Define a version of Young tableaux related to process $HAD_{k}$. 

We plan to address some of these  in subsequent work. The most important open problem, however, is {\it the complexity of computing LHS.}

This research has been supported by CNCS IDEI Grant PN-II-ID-PCE-2011-3-0981 "Structure and computational difficulty in combinatorial optimization: an interdisciplinary approach".


\newcommand{\etalchar}[1]{$^{#1}$}

\section*{Appendix} 

\subsection{Proof of Theorem~\ref{ineq}}

\begin{itemize} 
\item[1.] 

For $k\geq2,$ consider the sequence $X=[1,k+1,k,k-1,\cdots,2]$. 

\begin{lemma} We have
\[
MHS_{1}(X)=k, MHS_{2}(X)=k-1, \ldots, MHS_{k}(X)=1.
\]
\end{lemma} 
\begin{proof} 
Applying the Greedy algorithm we obtain the following heap decompositions: 
\begin{itemize}
\item ${\bf MHS_{1}(X)=k:}$ $H_1=[1,k+1]$, $H_2=[k]$, $H_3=[k-1]$, 
\ldots, $H_k=[2]$.
\item ${\bf MHS_{2}(X)=k-1:}$ $H_{1}=[1,k+1,k]$, $H_2=[k-1]$, $H_3=[k-2]$, \ldots, $H_{k-1}= [2]$.
\\ $\vdots$
\item ${\bf MHS_{i}(X)=k-i+1:}$ $H_{1}=[1,k+1,k,\ldots ,k-i+2]$, $H_2=[k-i+1]$, $H_3=[k-i]$, \ldots, $H_{k-i+1}= [k+2]$.
\\ $\vdots$
\item ${\bf MHS_{k}(X)=1:}$ $H_1=[1,k+1,k,\cdots,2].$ 
\end{itemize}
\end{proof}\qed
\item[2. ] 
Let $k,n\geq2$. Define sequence
\begin{eqnarray*}
X^{(k,n)} & = & [1, \\ 
 &  & (2+(k-1)),k,\dots,2,\\
 &  & (3+2(k-1)+(k-1)^{2}),(k+k^{2}),\dots,(2+k),\\
 &  & \vdots\\
 &  & \sum_{i=0}^{n}(n+1-i)(k-1)^{i},\dots,1+\sum_{i=0}^{n-1}(n-i)(k-1)^{i}]\\
\end{eqnarray*}
in other words
$X^{(k,n)}=[1,X_{1},X_{2},\dots,X_{n}],$ where for each $1\leq t\leq n$
the subsequence $X_{t}$ is $X_{t}=[\sum_{i=0}^{t}(t+1-i)(k-1)^{i},\,
,\dots,\,1+\sum_{i=0}^{t-1}(t-i)(k-1)^{i}]$. $X_{t}$ has 
 $(k-1)^{t}+(k-1)^{t+1}+\ldots + 1=\frac{(k-1)^{t+1}-1}{k-2} $ many elements. 

We can see that this sequence is $k$-heapable, thus $MHS_{k}(X)=1$: $|X_{t}|= (k-1) |X_{t-1}|+1< k|X_{t}|,$ and every number in $X_{t}$ is larger than every number in $X_{t-1}$. Thus we can arrange the $X_{t}$'s on (incomplete)  heap levels, with every node in $X_{t}$ a child of some node in $X_{t-1}$. 
\begin{theorem}
We have
\begin{eqnarray*}
MHS_{k-1}(X^{(k,n)}) & = & n+1.\\
\end{eqnarray*}
\end{theorem}
\begin{proof}

We apply the GREEDY algorithm. After sequence $X_{1}$ two $(k-1)$-heaps are created. $H_{1}$ has two full levels, $H_{2}$ contains only the root 2. Sequence $X_{2}$ has length $k^2$. $(k-1)^2$ elements go on the third level of $H_{1}$. $k-1$ elements go on the second level of $X_{2}$. The remaining $k^2-(k-1)^{2}-2(k-1)=1$ element starts a new heap $H_{3}$. 

By induction we easily prove the following
\begin{lemma}
For every $t\geq 1$, the $\frac{(k-1)^{t+1}-1}{k-2}$ elements of $X_{t}$ go via GREEDY as follows: 
\begin{itemize}
\item $(k-1)^{t}$ of them go on level $t$ of $H_{1}$, 
\item $(k-1)^{t-1}$ of them go on level $t-1$ of $H_{2}$,
\item $\ldots$
\item $k-1$ of them go on the first level of $H_{t}$.   
\end{itemize} 
The remaining $\frac{(k-1)^{t+1}-1}{k-2}-\sum_{i=1}^{t} (k-1)^{i}=1$ element starts a new heap $H_{t+1}$.  
\end{lemma} 
\end{proof}

\end{itemize}\qed

\subsection{Proof of Theorem~\ref{limit}}

{\bf First sequence:} Define $a_{n}$ to be the expected cardinality of the multiset of {\bf slots} (particles lifelines in process $HAD_{2}$)) at moment $n$.  Clearly $a_{n}/n=2l(n)-c(n)$. Also, given $Z=(Z_{0},Z_{1},\ldots, Z_{n-1})$ a finite trajectory in [0,1] {\bf and an initial set of slots $T$},
denote by $s(Z;T)$ the multiset of {\bf particles (slots) added during $Z$} that are still alive at the end of the trajectory $Z$, if {\bf at time $t=0$ the process started with the slots in $T$} (omitting the second argument if $T=\emptyset$), and $a(Z;T)=|s(Z;T)|$.
Finally denote by $v(Z;T)$ the submultiset of $s(Z;T)$ consisting of elements with multiplicity two, and by $l(Z;T)=|v(Z;T)|$.

Subadditivity of $a_{n}$ will follow from the fact that the  property holds {\it on each trajectory}: If $X=(X_{0},\ldots, X_{n-1})$ and $Y_{m}=(X_{n}\ldots X_{n+m-1})$ then in fact we can show that
\begin{equation} \label{ineq3}
 a(XY_{m})\leq a(X)+a(Y_{m}). 
\end{equation} 
Clearly $a_{n}=E_{|X|=n}[a(X)]$ so~(\ref{ineq}) implies that $a_{n}$ is subadditive. It turns out that, together with~(\ref{ineq3}), we will need to simultaneously prove that 
\begin{equation}
s(Y_{m}) \subseteq s(Y_{m};s(X))\mbox{ (as multisets)}
\label{ineq2}
\end{equation} 
We prove~(\ref{ineq3}) and~(\ref{ineq2}) by induction on $m=|Y_{m}|$. Clearly the inclusion is true if $m=0$. Let $Y_{m}=Y_{m-1}X_{n+m-1}$ and 
$s(XY_{m})=W_{m}\cup Z_{m}$, with $W_{m}=s(X)\cap s(XY_{m})$, $Z_{m}=Y_{m}\cap s(XY_{m})$.

$s(XY_{m})$ modifies $s(XY_{m-1})$ by adding two copies of $X_{n+m-1}$ to $W_{m}$ and, perhaps, erasing some $p_{m}$, the largest element (if any) in $s(XY_{m-1})$ smaller or equal to $X_{n+m-1}$. Thus $a(XY_{m})-a(XY_{m-1})\in \{1,2\}$. 
 
Similarly, $s(Y_{m})$ modifies $s(Y_{m-1})$ by adding two copies of $X_{n+m-1}$ and, perhaps, erasing some $r_{m}$, the largest element (if any) {\bf in $s(Y_{m-1})$} smaller or equal to $X_{n+m-1}$. Thus $a(Y_{m})-a(Y_{m-1})\in \{1,2\}$. 

All that remains in order to prove that $a(XY_{m})-a(Y_{m})\leq 
a(XY_{m-1})-a(Y_{m-1})$ (and thus establish inequality~(\ref{ineq}) inductively for $m$ as well) is that 
$(a(Y_{m})-a(Y_{m-1})=1) \Rightarrow (a(XY_{m})-a(XY_{m-1})=1).$ 
This follows easily from inductive hypothesis~(\ref{ineq2}) for $m-1$: if $a(Y_{m})-a(Y_{m-1})=1$ then some element in $s(Y_{m-1})$ is less or equal to $X_{n+m-1}$. The same must be true for $s(Y_{m-1};s(X))$ and hence for $s(XY_{m-1})$ as well 
(noting, though, that $p_{m}$ may well be an element of $X$).
Now we have to show that~(\ref{ineq2}) also remains true: clearly the newly added element, $X_{n+m-1}$, has multiplicity two in both $s(Y_{m})$ and $s(Y_{m};s(X))$. Suppose we erase some element $r_{m}$ from $s(Y_{m-1})$. Then $r_{m}$ belongs to 
$s(Y_{m-1};s(X))$, has multiplicity at least one there, and is {\it the largest element smaller or equal to $X_{n+m-1}$ in 
$s(Y_{m-1};s(X))\cap s(Y_{m-1})$}. Thus, when going from $s(Y_{m-1};s(X))$ to $s(Y_{m};s(X))$ we either erase one copy of 
$p_{m}$ or do not erase nothing (perhaps we erased some element in $s(X)$, which is not, however, in $s(Y_{m-1};s(X))$) Suppose, on the other hand that no element in $s(Y_{m-1})$ is smaller or equal to $X_{n+m-1}$. There may be such an erased element $p_{m}$ in $s(Y_{m-1};s(X))$, but {\it it certainly did not belong to $s(Y_{m-1})$}. In both cases we infer that relation 
$s(Y_{m}) \subseteq s(Y_{m};s(X))$ is true. 

{\bf Second sequence:} 

The proof is very similar to the first one: 
Define, in a setting similar to that of the first sequence, $u(X,T)$ to be the cardinality of the submultiset of $s(Z,T)$ 
of elements with multiplicity two. Define $a_{n}$ to be the expected number of elements with multiplicity two at stage $n$. That is, $a_{n}=E_{|X|=n}[u(X)]=l(n)-c(n)$. We will prove by induction on $m$ that if $X=(X_{0},\ldots, X_{n-1})$ and $Y_{m}=(X_{n}\ldots X_{n+m-1})$ then 
\begin{equation} \label{ineq4}
 u(XY_{m})\leq u(X)+u(Y_{m}). 
\end{equation} 

The result is clear for $m=0$. In the general case, $m\geq 1$, 
$v(XY_{m})$ modifies $v(XY_{m-1})$ by adding $X_{n+m-1}$ and, perhaps, erasing some $p_{m}$, the largest element (if any) in $s(XY_{m-1})$ smaller or equal to $X_{n+m-1}$ if this element is in $v(XY_{m-1})$. Thus $u(XY_{m})-u(XY_{m-1})\in \{0,1\}$. 
Similarly, $u(Y_{m})$ modifies $u(Y_{m-1})$ by adding $X_{n+m-1}$ and, perhaps, erasing some $r_{m}$, the largest element (if any) {\bf in $s(Y_{m-1})$}, if this element is smaller or equal to $X_{n+m-1}$. Thus $u(Y_{m})-u(Y_{m-1})\in \{0,1\}$. 

If $u(XY_{m-1})\leq u(X)-1+u(Y_{m})$ then clearly $u(XY_{m})-u(Y_{m})\leq u(X)$. The only problematic case may be when $u(XY_{m-1})-u(Y_{m-1})=u(X)$, $u(XY_{m})-u(XY_{m-1})=1$, 
$u(Y_{m})-u(Y_{m-1})=0$. But this means that $r_{m}$ exists (and is erased from $v(Y_{m-1})$). Since $s(Y_{m-1})\subseteq s(Y_{m-1};s(X))$, $r_{m}$ must be erased from $s(Y_{m-1};s(X))$. 
In other words, the bad case above cannot occur.

\subsection{Proof of Theorem~\ref{unif}}

We use essentially the classical proof based on the {\it hook walk} from 
\cite{hook-walk}, slightly adapted to our framework: Define for a heap  table $T$ with $n$ elements
\[
F_{T}=\frac{n!}{\prod_{(\alpha,i)\in dom(T)} H_{\alpha,i}}
\]
and $C(T)$, the set of {\it corners of $T$}, to be the set of cells $(\alpha,i)$ of $T$ with $H_{\alpha,i}=1$. Given $\gamma\in C(T)$ define 
$T_{\gamma}=T\setminus \{\gamma\}$. We want to prove that 
\begin{equation} 
\sum_{\gamma\in C(T)} \frac{F_{T_{\gamma}}}{F_{T}}\geq 1. 
\label{fin}
\end{equation} 
(of course, for $k=1$ we can actually prove equality in Formula~\ref{fin} above). 
This will ensure (by induction upon table size) the truth of our lower bound. 

We need some more notation: for $(\alpha,i)\in dom(T)$, denote 
\begin{equation} 
Heap_{\alpha,i}=\{(\beta,i)\in dom(T):\alpha \sqsubseteq \beta \}
\end{equation} 
the {\it heap hook of $(\alpha,i)$}, and by 
\begin{equation} 
Vec_{\alpha,i}=\{(\alpha,j)\in dom(T):i\leq j \}
\end{equation} 
its {\it vector hook} (thus $H_{\alpha,i}=|Heap_{\alpha,i}|+|Vec_{\alpha,i}|-1$).

By applying formulas for $F_{T},F_{T_{\gamma}}$ we get 
\begin{align}
\frac{F_{T_{\gamma}}}{F_T}& =\frac{1}{n}\cdot \prod_{\gamma\in Heap_{\beta,j}} \frac{H_{\beta,j}}{H_{\beta,j}-1}\cdot \prod_{\gamma\in Vec_{\beta,j}} \frac{H_{\beta,j}}{H_{\beta,j}-1} \nonumber \\
& = \frac{1}{n}\cdot \prod_{\gamma\in Heap_{\beta,j}} (1+\frac{1}{H_{\beta,j}-1})\cdot \prod_{\gamma\in Vec_{\beta,j}} (1+\frac{1}{H_{\beta,j}-1})\label{hw}
\end{align}

We consider the {\it hook walk on $T$}, defined in Figure~(\ref{hw:k}).

\begin{figure}
\begin{framed} 
\begin{itemize} 
\item Choose (uniformly at random) a cell $(\alpha_{1},i_{1})$ of $T$. 
\item let $i=1$. 
\item while (($\alpha_{i},t_{i})$ is not a corner of $T$): 
\item \hspace{5mm} Choose $(\alpha_{i+1},t_{i+1})$ uniformly at random from $H((\alpha_{i},t_{i}))\setminus \{(\alpha_{i},t_{i})\}$.
\item \hspace{5mm} Let $i=i+1$. 
\item Return corner $(\alpha_{n},i_{n})$. 
\vspace{-0.7cm}
\end{itemize}
\end{framed} 
\vspace{-0.5cm} 
\caption{The hook walk.}
\label{hw:k}
\end{figure} 

Interpret terms from the product in formula~(\ref{hw}) as probabilities of paths in the hook walk, ending in corner $\gamma$, as follows: 
\begin{itemize}
\item Choose $(\alpha_{1},i_{1})$ uniformly at random from $T$ (i.e. with probability 1/n).
\item Terms $(\beta,i)$ in the first product whose contribution is $\frac{1}{H_{\beta,i}-1}$ correspond to cells where the walk makes "hook moves" towards $\gamma$. 
\item Terms $(\beta,i)$ in the second product whose contribution is $\frac{1}{H_{\beta,i}-1}$ correspond to cells where the walk makes "vector moves" towards $\gamma$. 
\end{itemize} 
Indeed, consider a path $P:(\alpha,i):= (\alpha_{1},i_{1})\rightarrow (\alpha_{2},i_{2})\rightarrow \ldots \rightarrow (\alpha_{n},i_{n})=\gamma$. Define its {\it hook projection} to be 
set $A=\{\alpha_{1},\alpha_{2}, \ldots, \alpha_{n}\}$ and its vector projection to be the set $B=\{i_{1},i_{2},\ldots, i_{n}\}$. 

Just as 
in \cite{hook-walk}, given set of words $A=\{\alpha_{1},\ldots \alpha_{m}\}$, with $\alpha_{1}=\alpha$ and $\alpha_{i}\sqsubset \alpha_{i+1}$ and set of integers $B=\{i_{1},\ldots, i_{r}\}$ with $i_{1}=i$ and $i_{l}<i_{l+1}$, 
the probability $p(A,B)$ that the hook walk  has the hook(vector) projections $A(B)$ (thus starting at $(\alpha_{1}, i_{1})$) is 
\begin{equation} 
P(A,B)\leq \prod_{\beta \in A, \beta \neq \alpha_{m} } (1+\frac{1}{H_{\beta,i_{r}}-1})\cdot \prod_{i\in B, i\neq i_{r}} (1+\frac{1}{H_{\alpha_{m},i}-1})
\label{eq-pi}
\end{equation} 
Indeed, as in \cite{hook-walk} 
\begin{align}
P(A,B) & =\frac{1}{H_{\alpha_{1},i_{1}}-1} [ P(A-\{\alpha_{1}\},B)+ P(A,B-\{i_{1}\})]\leq \nonumber\\ 
                & \leq \frac{1}{H_{\alpha_{1},i_{1}}-1} [(H_{\alpha_{1},i_{r}}-1)+ (H_{\alpha_{m},i_{1}}-1)]\cdot (RHS)
\label{eq-p}
\end{align}
where $(RHS)$ is the right-hand side product in equation~(\ref{eq-pi}), and in the second row we used the inductive hypothesis.  

For $k=1$, in \cite{hook-walk} we would use an equality of type $H_{\alpha_{1},i_{1}}-1=(H_{\alpha_{1},i_{r}}-1)+ (H_{\alpha_{m},i_{1}}-1)$. For $k\geq 2$ such an equality is no longer true, and we only have inequality 
\begin{equation} 
H_{\alpha_{1},i_{1}}-1 \geq (H_{\alpha_{1},i_{r}}-1)+ (H_{\alpha_{m},i_{1}}-1)
\label{just-ineq}
\end{equation} 
leading to a proof of equation~(\ref{eq-pi}). 

To justify inequality~(\ref{just-ineq}), note that, by property (b) of heap tableaux, since $\alpha_{1}\sqsubset \alpha_{m}$, 
\begin{equation} 
|Vec(\alpha_{m},i_{1})|\leq |Vec(\alpha_{1},i_{1})|
\label{eq-vec}
\end{equation} 
On the other hand 
\begin{equation}
|Heap(\alpha_{1},i_{1})|\geq |Heap(\alpha_{1},i_{r})|+(|Heap(\alpha_{m},i_{1})|-1).
\label{eq-heap}
\end{equation} 
This is true by monotonicity property (c) of heap tableaux: every path present in the heap $H_{r}$ rooted at $(\alpha_{1},i_{r})$ is also present in the heap $H_{1}$ rooted at $(\alpha_{1},i_{1})$. Heap $H_{r}$ is empty below node $\gamma=(\alpha_{m},i_{r})$, but $H_{1}$ contains the subheap rooted at $(\alpha_{1},i_{r})$ (of size $|Heap(\alpha_{1},i_{r})|-1)$ {\it any maybe some other subheaps,} rooted at nodes $w\in H_{1}$ whose correspondent in $H_{r}$ has no descendents.  
Summing up equations~(\ref{eq-vec}) and~(\ref{eq-heap}) we get our desired inequality~(\ref{just-ineq}). Example in Figure ~\ref{youngContraexemplu} shows that inequality ~(\ref{just-ineq}) can be strict: The hook length of $H_{1,\lambda}-1 = 7$ but $H_{2,\lambda}-1=2-1$ and $H_{1,0}-1=2-1$. The reason is that the grayed cells are not counted in the hook of $(1,0)$, but they belong to the hook of $(1,\lambda)$. 

\begin{figure}[ht]
\begin{center} 
\begin{minipage}{.5\textwidth}

\[
\gyoung(::\lambda:0:1:<00>:<1^2>:<\cdots>:<1^3>:<\cdots>:<1^4>:<\cdots>:<1^5>,:1;8!\blue2!\white!\grey5!\white;<\perp>!\grey4!\white;<\perp>!\grey3!\white;<\perp>!\grey2!\white;<\perp>!\grey1!\white,:2!\blue2!\white;1)
\]

\end{minipage}  
\end{center}
\caption{Example showing that inequality (\ref{just-ineq}) is strict.}
\label{youngContraexemplu}
\end{figure}

Finally, adding up suitable inequalities~(\ref{eq-pi}) we infer that $s_{\gamma}$, the probability that the walk ends up at $\gamma$, equal to  
\[
s_{\gamma}=\frac{1}{n}\sum p(A,B)
\]
(the sum being over all suitable sets $A,B$) is less or equal than the expansion~(\ref{hw}) of $\frac{F_{T_{\gamma}}}{F_{T}}$. Since the sum of probabilities adds up to 1, inequality~(\ref{fin}) follows.  

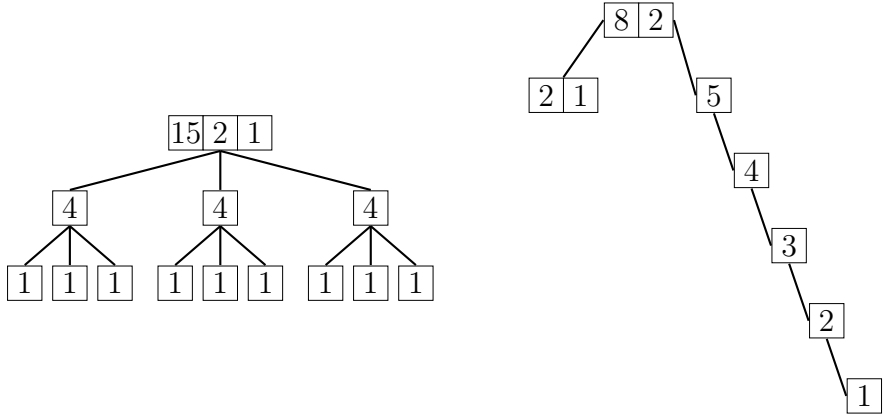
\begin{figure}[ht]
\begin{minipage}{.5\textwidth}
\begin{tikzpicture}
\node[inner sep=0pt] (A) at (3,1)
    {\gyoung(;<15>;2;1)};
\node[inner sep=0pt] (B) at (1,0)
    {\gyoung(;4)};
\node[inner sep=0pt] (C) at (3,0)
    {\gyoung(;4)};
\node[inner sep=0pt] (D) at (5,-0)
    {\gyoung(;4)};
    
\node[inner sep=0pt] (E) at (0.4,-1)
    {\gyoung(;1)};
\node[inner sep=0pt] (F) at (1,-1)
    {\gyoung(;1)};
\node[inner sep=0pt] (G) at (1.6,-1)
    {\gyoung(;1)};
    
\node[inner sep=0pt] (H) at (2.4,-1)
    {\gyoung(;1)};
\node[inner sep=0pt] (I) at (3,-1)
    {\gyoung(;1)};
\node[inner sep=0pt] (J) at (3.6,-1)
    {\gyoung(;1)};
    
\node[inner sep=0pt] (K) at (4.4,-1)
    {\gyoung(;1)};
\node[inner sep=0pt] (L) at (5,-1)
    {\gyoung(;1)};
\node[inner sep=0pt] (M) at (5.6,-1)
    {\gyoung(;1)};
\draw[-,thick] (A.south) -- (B.north);
\draw[-,thick] (A.south) -- (C.north);
\draw[-,thick] (A.south) -- (D.north);

\draw[-,thick] (B.south) -- (E.north);
\draw[-,thick] (B.south) -- (F.north);
\draw[-,thick] (B.south) -- (G.north);

\draw[-,thick] (C.south) -- (H.north);
\draw[-,thick] (C.south) -- (I.north);
\draw[-,thick] (C.south) -- (J.north);

\draw[-,thick] (D.south) -- (K.north);
\draw[-,thick] (D.south) -- (L.north);
\draw[-,thick] (D.south) -- (M.north);
\end{tikzpicture}
\end{minipage} 
\begin{minipage}{.5\textwidth}
\begin{tikzpicture}
\node[inner sep=0pt] (A) at (4,1.5)
    {\gyoung(;8;2)};
\node[inner sep=0pt] (B) at (3,0.5)
    {\gyoung(;2;1)};
\node[inner sep=0pt] (C) at (5,0.5)
    {\gyoung(;5)};
\node[inner sep=0pt] (D) at (5.5,-0.5)
    {\gyoung(;4)};
\node[inner sep=0pt] (E) at (6,-1.5)
    {\gyoung(;3)};
\node[inner sep=0pt] (F) at (6.5,-2.5)
    {\gyoung(;2)};
\node[inner sep=0pt] (G) at (7,-3.5)
    {\gyoung(;1)};
\draw[-,thick] (A.west) -- (B.north);
\draw[-,thick] (A.east) -- (C.west);
\draw[-,thick] (C.south) -- (D.west);
\draw[-,thick] (D.south) -- (E.west);
\draw[-,thick] (E.south) -- (F.west);
\draw[-,thick] (F.south) -- (G.west);
\end{tikzpicture}
\end{minipage}

\caption{(a). Example $T_{3,3}$. (b). Counterexample $W_{4}$. The hook formula is tight for heap tableau (a). but not tight for (b). In both cases cell contents represent the hook lengths.} 
\label{hookex}
\end{figure}

Let us now deal with examples/counterexamples.  

First we present a set of arbitrarily large heap tableaux, different from both heap-ordered trees and Young tableaux, for which the hook inequality is tight: for $r\geq 2, k\geq 1$ consider heap  table $T_{r,k}$ (Fig.~\ref{hookex}(a)) to have $n=S_{k,r}+k-1$ nodes, distributed in a complete $k$-ary tree $H_{1}$ with $r$ levels $0,1,\ldots r-1$ and $S_{k,r}$ nodes, and then $k-1$ one-element heaps $H_{2},\ldots, H_{k}$. We employ notation 
\[
S_{k,l}=1+k+\ldots + k^{l-1}=\frac{k^{l}-1}{k-1}
\]
The number of ways to fill up such a heap tableau is ${{n-1}\choose {k-1}}\cdot N_{k,r}$, where $N_{k,r}$ is the number of ways to fill up a complete $k$-ary tree with $r$ levels. 
\[
N_{k,r}=\frac{S_{k,r}!}{\prod_{i=0}^{r-1} (S_{k,r-i})^{k^{i}}}
\]
This happens because for every subset $A$ of $\{2,\ldots ,n\}$ of cardinality $k-1$, element 1 together with those not in $A$ can be distributed in $H_{1}$ in $N_{k,r}$ ways. 

Putting all things together, the total number of fillings of $T_{r,k}$ is 
\[
\frac{(S_{k,r}+k-2)!\cdot (S_{k,r})!}{(k-1)!\cdot (S_{k,r}-1)!\cdot S_{k,r}\cdot \prod_{i=1}^{r-1} (S_{k,r-i})^{k^{i}}}= \frac{(n-1)!}{(k-1)!\cdot \prod_{i=1}^{r-1} (S_{k,r-i})^{k^{i}}}
\]

Hook lengths are $1,2,\ldots, k-1$ (for the nodes in the one-element heaps), $(S_{k,r-i})^{k^{i}}$ (for the non-root nodes in $H_{1}$) and $n$ (for the root node of $H_{1}$). The resulting formula

\begin{equation} 
\frac{n!}{(k-1)!\cdot n\cdot \prod_{i=1}^{r-1} (S_{k,r-i})^{k^{i}}}= \frac{(n-1)!}{(k-1)!\cdot \prod_{i=1}^{r-1} (S_{k,r-i})^{k^{i}}}
\end{equation}

is the same as the total number computed above. 

Now for the counterexamples: consider heap tableaux $W_{r}$ (Fig.~\ref{hookex}(b), identical to the heap tableau in Fig.~\ref{youngContraexemplu}) defined as follows: $W_{r}$ consists of two heaps, $H_{1}$ with cells with addresses
$(1,\lambda), (1,0)$, $(1,1),(1,11), \ldots,$ $(1,1^{2r-3})$, and $H_{2}$ with cells with addresses $(2,\lambda),(2,0)$. $W_{r}$ has $n=2r+1$ nodes. 

Hook values of cells in $H_{1}$ are $2r,2,2r-3,2r-4, \ldots, 1$. 
Hook values of cells in $H_{2}$ are $2,1$, respectively. Thus the hook formula predicts 
\[
\frac{(2r+1)!}{2\cdot 2\cdot 2r \cdot (2r-3)\cdot (2r-4)\cdot \ldots \cdot 1}=\frac{(2r+1)(2r-1)(2r-2)}{4}
\]
ways to fill up the table. If $r$ is even then the number above is {\bf not} an integer, so the hook formula cannot be exact for these tableaux.

\subsection{Proof of Theorem~\ref{sch-heap}}

We prove that inserting a single integer element $x$ into a heap tableau $T$ results in another heap tableau $T\leftarrow x$. Therefore inserting a permutation $X$ will  result in a heap tableau. 

By construction, when an element is appended to a vector, the vector remains increasing. Also, if an element $y$ bumps another element $z$ from a vector $V$ (presumed nondecreasing) then $z$ is the smallest such element in $V$ greater than $y$. Thus, replacing $z$ by $y$ preserves the nondecreasing nature of the vector $V$.

All we need to verify is that min-heap invariant (b) (initially true for the one-element heap tableau) also remains true when inserting a new element $x$. 

The case when $x$ is appended to $V_{\lambda}$ is clear: since invariant (b) was true before inserting $x$ for every address $r$ we have $|V_{\lambda}|\geq |V_r|.$ See the example above when we append $x=9.$
Thus what we are doing, in effect, by appending $x$ to $V_{\lambda}$ is start a new heap. 

Suppose instead that inserting $x$ bumps element $x_{1}$ from $V_{\lambda}$. 
Necessarily $x<x_{1}$. Suppose $i$ is the position of $x_{1}$ in $V_{\lambda}$, that is $x_{1}$ was the root of heap $H_{i}$. By reducing the value of the root, the heap $H_{i}$ still verifies the min-heap invariant. Now suppose $x_{1}$ bumps element $x_{2}$. We claim that $x_{2}$ has rank at most $i$ in its vector. Indeed, the element with rank $i$ in the vector of $x_{2}$ was larger than $x_{1}$ (by the min-heap property of $H_{i}$). So $x_{2}$ must have had rank at most $i$. Let $j$ be this rank. 

\begin{figure}[ht]
\begin{center} 
\begin{minipage}{.35\textwidth}

\begin{tikzpicture}
\node[inner sep=0pt] (A) at (4,2)
   	{\gyoung(:::::<\downarrow x>::::,;;;;<\cdots>!\grey<x_1>!\white;<\cdots>;,:::::<i>::::)};
\node[inner sep=0pt] (B) at (2,0)
    {\gyoung(:::<\downarrow x_1>:::,;;<\cdots>!\grey<x_2>!\white;<\cdots>;;\cdots,:::<j>::<i>::)};
\node[inner sep=0pt] (C) at (6,0.4)
    {\gyoung(:)};
\node[inner sep=0pt] (D) at (1,-1.3)
    {\gyoung(:)};
\node[inner sep=0pt] (E) at (3,-1.5)
    {\gyoung(::\cdots,::)};
\node[inner sep=0pt] (F) at (2,-3)
    {\gyoung(:::<\downarrow x_{n}>:::,;;<\cdots>!\grey<x_n>,:::<s>)};
\node[inner sep=0pt] (G) at (4,-2.5)
    {\gyoung(:)};
    
\draw[-,thick] (A.south) -- (B.north);
\draw[-,thick] (A.south) -- (C.north);
\draw[-,thick] (B.south) -- (D.north);
\draw[-,thick] (B.south) -- (E.north);
\draw[-,thick] (E.south) -- (F.north);
\draw[-,thick] (E.south) -- (G.north);
\end{tikzpicture}

\end{minipage} 
\end{center}
\caption{Inserting $x$ and the bumps it determines.} 
\label{insert1} 
\end{figure}
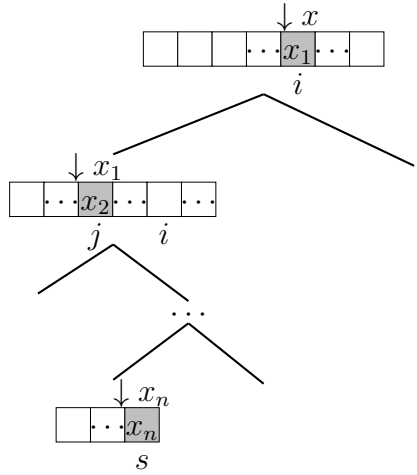

Since $x_{1}<x_{2}$, by replacing $x_{2}$ by $x_{1}$ the min-heap property is satisfied "below $x_{2}/x_{1}$". It is satisfied "above $x_{2}$" as well, since the parent of  $x_{2}$ either was (and still is) 
the root of $H_{j}$, a number less or equal to $x_{1}$ (in case $j<i$) or is $x$ (in case the rank of $x_{2}$ is exactly $i$). 

If $x_{2}$ bumps $x_{3},\ldots, $ etc we repeat the argument above on the corresponding sub-min-heap tableau. 

Suppose, finally, that element $x_{n}$, bumped from $V_{\alpha}$ by $x_{n-1}$, is appended to vector $V_{\beta}$. Let $s$ be the index of $x_{n}$ in $V_{\beta}$. We claim that $|V_{\alpha}|\geq s$. 

Indeed, $x_{n}$ is larger than the first $s-1$ elements of $|V_{\beta}|$. By the min-heap property, it is also larger than the initial $s-1$ elements of $|V_{\alpha}|$ as well. So its index in $V_{\alpha}$ before getting bumped could not have been less than $s$. That means that appending $x_n$ does not violate the min-heap invariant (b).

\subsection{Proof of Theorem~\ref{rs-heap}}

Given permutation $\sigma\in S_{n}$, denote by $P_{\sigma}$ the heap table obtained by applying the Schensted-HEAP$_{k}$ algorithm. 

Define a second heap table $Q_{\sigma}$ as follows: whenever we insert $\sigma(i)$ into $P$, we record the resulting sequence of bounces and {\bf insert i at the last place involved in the bounces.} 

\begin{example} Let $k=2$ and consider the permutation $\sigma=\begin{pmatrix}
1 & 2 & 3 & 4 & 5 & 6 \\
4 & 2 & 6 & 3 & 5 & 1 
\end{pmatrix}.$ 
The two corresponding heap tableaux are constructed below. For drawing convenience, during the insertion process they are not displayed in the heap-like form, but rather in the more compact Young-table equivalent format.  The resulting heap-tableaux are displayed in Figure~\ref{rs-result}. 

\begin{center} 
\begin{tabular}{|c|c|c|c|c|c|c|}
  \hline \hline
 &  $\leftarrow$ 4 & $\leftarrow$ 2 & $\leftarrow$ 6 & $\leftarrow$ 3 & $\leftarrow$ 5 & $\leftarrow$ 1 \\
 \hline \hline
$P_{\sigma}$: & $\gyoung(::\lambda,:1!\grey<4>)$ & $\gyoung(::\lambda:0,:1!\grey<2>;4)$ & $\gyoung(::\lambda:0,:1;2<4>,:2!\grey 6)$ & $\gyoung(::\lambda:0,:1;2;4,:2!\grey<3>;6)$ & $\gyoung(::\lambda:0,:1;2;4,:2;3;6,:3!\grey<5>)$ &  $\gyoung(::\lambda:0:1:,:1!\grey<1>!\white;4!\grey<2>!\white,:2;3;6,:3;5)$\\ 
\hline \hline
$Q_{\sigma}$: & $\gyoung(::\lambda,:1!\grey<1>)$ & $\gyoung(::\lambda:0,:1!\grey<1>;2)$ & $\gyoung(::\lambda:0,:1;1<2>,:2!\grey 3)$ & $\gyoung(::\lambda:0,:1;1;2,:2!\grey<3>;4)$ & $\gyoung(::\lambda:0,:1;1;2,:2;3;4,:3!\grey<5>)$ &  $\gyoung(::\lambda:0:1:,:1!\grey<1>!\white;2!\grey<6>!\white,:2;3;4,:3;5)$\\ 
  \hline \hline
\end{tabular}
\end{center} 
 
\end{example} 

\begin{figure}[ht]

\begin{minipage}{.5\textwidth}
\begin{tikzpicture}
\node[inner sep=0pt] (A) at (4,0)
    {\gyoung(1;3;5)};
\node[inner sep=0pt] (B) at (2,-1)
    {\gyoung(4;6)};
\node[inner sep=0pt] (C) at (6,-1)
    {\gyoung(2)};
\draw[-,thick] (A.south) -- (B.north);
\draw[-,thick] (A.south) -- (C.north);
\end{tikzpicture}
\end{minipage} 
\begin{minipage}{.5\textwidth}
\begin{tikzpicture}
\node[inner sep=0pt] (A) at (4,0)
    {\gyoung(1;3;5)};
\node[inner sep=0pt] (B) at (2,-1)
    {\gyoung(2;4)};
\node[inner sep=0pt] (C) at (6,-1)
    {\gyoung(6)};
\draw[-,thick] (A.south) -- (B.north);
\draw[-,thick] (A.south) -- (C.north);
\end{tikzpicture}
\end{minipage} 
\caption{(a). Heap-tableau $P_{\sigma}$. (b). (Standard) heap-tableau $Q_{\sigma}$.} 
\label{rs-result} 
\end{figure}
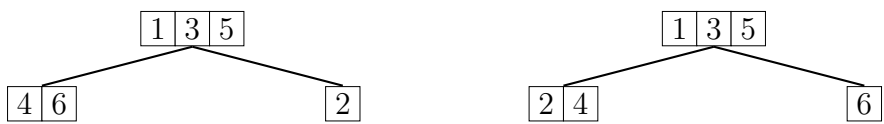

There are two things to prove about the algorithm outlined above:
\begin{itemize} 
\item[(i). ] For every permutation $\sigma$, $Q_{\sigma}$ is a heap tableau of the same shape as heap tableau $P_{\sigma}$. Moreover,
\begin{lemma} 
 $Q_{\sigma}$ is a heap tableau in standard form.  
\end{lemma} 
\item[(ii). ] One can uniquely identify permutation $\sigma$ from the pair $(P_{\sigma},Q_{\sigma})$. 
\end{itemize} 

\begin{itemize} 
\item[(i). ] The fact that the shape is the same is easy: whenever number $\sigma(i)$ is inserted into $P_{\sigma}$, this table changes by exactly one (filled) position. When $i$ is inserted into $Q_{\sigma}$, the position on which it is inserted is the unique position that was added to 
$P_{\sigma}$: the position of the final insertion after a (perhaps empty) sequence of bumps. Therefore the two heap tableaux have the same shape throughout the process, and at the end of it. 

Let us show now that $Q_{\sigma}$ is a heap tableau. We will show that invariants (b),(c). remain true throughout the insertion process. 

They are, indeed, true at the beginning when $Q_{\sigma}=[1]$. Proving the heap invariant (b). is easy: numbers are inserted into $Q_{\sigma}$ in the order $1,2,\ldots, n$. Each number is, therefore, larger than any number that is an ancestor in its heap. As each number $i$ is inserted as a leaf in its corresponding heap, all heap conditions are still true after its insertion. 

The vector invariant (c). is equally easy: number $i$ is appended to an old vector or starts a new one. The second case is trivial. In the first one $i$ is the largest number inserted so far into $Q_{\sigma}$, therefore the largest in its vector. 

Finally, the fact that $Q_{\sigma}$ is a standard tableau follows from the Algorithm: Schensted-HEAP$_{k}$ starts a new vector from the leftmost position available. Therefore when it starts a new vector, its siblings to the left have acquired a smaller number, as they were already created before that point. Also, when it starts a new vector, all the vectors on the level immediately above have been created (otherwise Schensted-HEAP$_{k}$ would have started a new vector there) and have, thus, acquired a smaller number. 

\item[(ii). ] This is essentially the same proof ideea as that of the Robinson- \\ Schensted correspondence for ordinary Young tableaux: given heap tableaux $P,Q$ with the same shape we will recover the pairs $(n,\sigma(n))$, $(n-1,\sigma(n-1))$, $\ldots, (1,\sigma(1))$ in this backwards order by reversing the sequences of bumps. We will work in the more general setting when $P$ contains $n$ distinct numbers, not necessarily those from $1$ to $n$. On the other hand, since $Q$ is standard, $Q$ will contain these numbers, each of them exactly once. 

The result is easily seen to be true for $n=1,n=2$. From now on we will assume that $n\geq 3$ and reason inductively. 

Suppose $n$ is in vector $V_{\lambda}$ of $Q_{\sigma}$. Then the insertion of $\sigma(n)$ into $P_{\sigma}$ did not provoke any bumps. $\sigma(n)$ is the integer in vector $V_{\lambda}$ of $P_{\sigma}$ sitting in the same position as $n$ does in $Q_{\sigma}$. Suppose, on the other hand, that $n$ is in a different vector of $Q_{\sigma}$. Then $n$ is the outcome of a series of bumps, caused by the insertion of $\sigma(n)$. 

Let $x$ be the integer in $P_{\sigma}$ sitting at the same position as $n$ in $Q_{\sigma}$. Then $x$ must have been bumped from the parent vector in the heap-table by some $y$. $y$ is uniquely identified, as the largest element smaller than $x$ in that vector. There must exist a smaller element in that vector by the heap invariant, so $y$ is well-defined. Now $y$ must have been in turn bumped by some $z$ in the parent vector. We identify $z$ going upwards, until we reach vector $V_{\lambda}$, identifying element $\sigma(n)$. 

\begin{example} 
Consider, for example the case of $n=6$ in Figure~\ref{rs-result}. Element $2$ in $P_{\sigma}$ (sitting in the corresponding position) must have been bumped by 1  in the top row. Therefore $\sigma(6)=1$. 
\end{example}

Now we delete $\sigma(i),i$ from the two heap tableaux and proceed inductively, until we are left with two tables with one element, identifying permutation $\sigma$ this way. 

What allows us to employ the induction hypothesis is the following

\begin{lemma}
Removing the largest element $n$ from a standard heap tableau $T$ yields another standard heap tableau.
\label{red}  
\end{lemma} 

\begin{proof} 
Suppose $n$ is in a vector of length at least two. Clearly, by removing $n$ all the vectors in the heap remain the same, so the resulting table 
is standard. 

Suppose, therefore, that $n$ is the only element in a vector $V_{\beta}$ of $T$, $\beta=zb$, $b\in \Sigma_{k}$. Since $T$ was standard, all the left sibling vectors $V_{za}$ of $V$ ($a\in \Sigma_{k}, a<b$) are nonempty, and all the vectors on previous levels of $T$ are nonempty.

Removing $V$ preserves these properties (its leftmost sibling becomes the last vector, or the level disappears completely). 

\end{proof}\qed
Completing the proof of Lemma~\ref{red} also completes the proof of Theorem~\ref{rs-heap}. 
\end{itemize} 
\qed

\end{document}